\documentclass[11p,reqno]{amsart}
\usepackage{amsmath}
\usepackage{amsfonts}
\usepackage{amssymb}
\usepackage{graphicx}
\usepackage{color}
\usepackage{tikz-cd}
\newtheorem{theorem}{Theorem}[section]
\newtheorem*{theorem*}{Theorem}
\newtheorem{lemma}{Lemma}[section]
\newtheorem{corollary}[theorem]{Corollary}
\newtheorem{proposition}{Proposition}[section]

\def\Ric{\text{Ric}}

\def\aint{\frac{\ \ }{\ \ }{\hskip -0.4cm}\int}

\def\Ric{\operatorname{Ric}}

\numberwithin{equation}{section}

\begin{document}
	\title[$\Ric_k<0$ manifolds]{An application of a $C^2$-estimate for a complex Monge-Amp\`ere equation}
	
\author{Chang Li}
\address{Chang Li. Hua Loo-Keng Center for Mathematical Sciences, Academy of Mathematics and Systems Science, Chinese Academy of Sciences, Beijing, P.R.China, 100190}
\email{chang\_li@pku.edu.cn}

\author{Lei Ni}
\address{Lei Ni. Department of Mathematics, University of California, San Diego, La Jolla, CA 92093, USA}
\email{leni@ucsd.edu}

\author{Xiaohua Zhu}

\address{Xiaohua Zhu. School of Mathematical Science, Peking University, Beijing, China}
\email{xhzhu@math.pku.edu.cn}


\begin{abstract} By studying a complex Monge-Amp\`ere equation, we present an alternate proof to a recent result of Chu-Lee-Tam concerning the projectivity of a compact K\"ahler manifold $N^n$ with $\Ric_k< 0$ for some integer $k$ with $1<k<n$,  and   the ampleness of the canonical line bundle $K_N$.

\bigskip

\centerline{ {\it Dedicated to the 110th anniversary of S. S. Chern with a great honor.} }

\medskip

\end{abstract}


\maketitle

\section{Introduction}
In a recent preprint \cite{Chu-Lee-Tam}, the authors proved the following result.

\begin{theorem}[Chu-Lee-Tam] \label{thm:main1}Assume that $(N^n, \omega_g)$ is a compact K\"ahler manifold ($n=\dim_{\mathbb{C}}(N)$) with $\Ric_k(X, \overline{X})\le -(k+1)\sigma |X|^2$ for some $\sigma\ge0$. Then $K_N$ is nef and is ample if $\sigma>0$.
\end{theorem}

The curvature notion  $\Ric_k$ is defined as the Ricci curvature of the $k$-dimensional subspaces of the holomorphic tangent bundle $T'N$. Hence  it coincides with the holomorphic sectional curvature $H(X)$ when $k=1$, and with the Ricci curvature when $k=n=\dim_{\mathbb{C}}(N)$. The condition $\Ric_k>0$ is significantly different from its Riemannian analogue, i.e. the so-called $q$-Ricci of Bishop-Wu \cite{BC}, since it exams only the Ricci curvature of the  subspaces in the holomorphic tangent space $T'N$, thus unlike its Riemannian analogue,  $\Ric_k>0$ ($<0$) does not imply $\Ric_{k+1}>0$ ($<0$).  The study of the condition  of $\Ric_k<0$ was initiated by the second author \cite{Ni-cpam} to generalize the hyperbolicity of Kobayashi to the $k$-hyperbolicity of  a compact K\"ahler manifold. It is closely related to the degeneracy of holomorphic mappings from $\mathbb{C}^k$ into concerned manifolds (cf. Theorem 1.3 of \cite{Ni-cpam}). In fact a generalization to Royden's result was proved there. Moreover it was proved  recently by the second author that $\Ric_k>0$ implies that $M$ is projective and rational connected. The above result of Chu-Lee-Tam answers a question raised by the second author in \cite{Ni-cpam}, namely  the projectivity of a compact K\"ahler manifold with $\Ric_k<0$ affirmatively. In view of the fact that $\Ric_1<0$ is the same as the holomorphic section curvature $H<0$, Theorem \ref{thm:main1} generalizes the earlier work of \cite{Wu-Yau, Tosatti-Yang}.

The proof of \cite{Chu-Lee-Tam} is via the study of a twisted K\"ahler-Ricci flow. In this note we provide a direct alternate proof  via the Aubin-Yau solution \cite{Aub,Yau} to a complex Monge-Amp\`ere equation (which is similar to the equation for the K\"ahler-Einstein metric in the negative first Chern class). This was the method utilized in \cite{Wu-Yau}. Comparing with \cite{Wu-Yau} here a  modification on the Monge-Amp\`ere equation  (cf. (\ref{MA-equ})) is necessary to adapt the method to the curvature condition $\Ric_k<0$ when $n>k>1$.  This makes the derivation of the estimates more involved.  The method here also extends to the more general setting considered in \cite{Chu-Lee-Tam}.

It was proved in \cite{Ni-Zheng2} that any compact K\"ehler manifold with the second scalar curvature $S_2>0$ (simply put $S_k$ is the average of $\Ric_k$) must be  projective. It remains an interesting question if $S_2<0$, or $\Ric^\perp_k<0$, also implies the projectivity.
For more backgrounds and references related to the theorem  please refer to \cite{Ni, Ni-cpam, Chu-Lee-Tam}.  One can also find the definitions and motivations of several other curvature notions, including $S_2, \Ric^\perp_k$, and problems related to them in \cite{Ni}. The geometry of $\Ric_k<0$ compares sharply with the case of $\Ric_k>0$. This is discussed at the end of Section 3 and in the Appendix via a construction of J. M$^{\mbox{c}}$Kernan.

\section{Preliminaries}
Since the result is known for $k=1, n$ below we  assume that $1<k<n$ in this section and the next.
Here we collect some algebraic estimates as consequences of the assumption $\Ric_k(X, \overline{X})\le -(k+1)\sigma |X|^2, \forall\,  (1,0)$-type tangent vector $X$. They are useful in obtaining key estimates for a Monge-Amp\`ere equation (cf. (\ref{MA-equ}) below).  The first is Lemma 2.1 of \cite{Chu-Lee-Tam}.

\begin{lemma}\label{lemma:Lee-Tam} Under the assumption that $\Ric_k(X, \overline{X})\le -(k+1)\sigma |X|^2$ the following estimate holds
\begin{equation}\label{eq:23}
(k-1) |X|^2 \Ric(X, \overline{X})+(n-k)R(X,\overline{X}, X, \overline{X})\le -(n-1)(k+1)\sigma |X|^4.
\end{equation}
\end{lemma}

The result follows by summing $ \Ric_k(X, \overline{X})\le -(k+1)\sigma |X|^2$ (the assumption) for a suitable chosen unitary basis.
By using a Royden's trick \cite{Roy} the following result was derived out of Lemma \ref{lemma:Lee-Tam} (cf.
 \cite[Lemma 2.2]{Chu-Lee-Tam}).
\begin{lemma}\label{Riemann-curvature-estimate}Let $(N, \omega)$ be a compact K\"ahler manifold with
\begin{align}\label{k-ricci}\mathrm{Ric}_k(X,\overline{X}) \leq -(k+1)\sigma|X|^2
\end{align}
 for some $\sigma \geq 0$.  Let $\tilde \omega=\omega_{\tilde g}$ be another K\"ahler metric on $N$. Set
$$G=\mathrm{tr}_{\tilde{\omega}}\omega.$$
Then  the following estimate holds
\begin{equation}\label{eq:22}
 2 \tilde{g}^{i \bar{j}} \tilde g^{k \bar{l}} \mathrm{R}_{i \bar{j} k \bar{l}}
\leq  \frac{-(n-1)(k+1)\sigma}{n-k}\left(G^2+|g|_{\tilde{g}}^{2}\right)
-\frac{k-1}{n-k} G \cdot \mathrm{tr}_{\tilde{g}} \mathrm{Ric}
-\frac{k-1}{n-k} \left\langle \omega, \mathrm{Ric} \right\rangle_{\tilde{g}}.
\end{equation}
\end{lemma}
\begin{proof} Here we provide a proof using the averaging technique (cf. Appendix of \cite{Ni-cpam}) instead of Royden's trick since the argument is more transparent. Pick a normal frame $\{\frac{\partial}{\partial z^i}\}$ so that $\tilde{g}_{i\bar{j}}=\delta_{ij}$ and $g_{i\bar{j}}=|\lambda_i|^2 \delta_{ij}$. Then $\{\frac{\partial}{\partial w^i}\}$, with $\frac{\partial}{\partial w^i}:=\frac{1}{\lambda_i}\frac{\partial}{\partial z^i}$, is a unitary frame for $\omega$. Lemma \ref{lemma:Lee-Tam} implies that (Einstein convention applied)
\begin{eqnarray*}&\,&
2(n-k)R^{\omega}_{i\bar{i}j\bar{j}}|\lambda_i|^2|\lambda_j|^2+(k-1)G\Ric^{\omega}_{s \bar{s}} |\lambda_s|^2+(k-1)\Ric^{\omega}_{i\bar{i}}|\lambda_i|^4 \\
&=& n(n+1)\aint_{\mathbb{S}^{2n-1}} (k-1) |Y|^2 \Ric^\omega(Y, \overline{Y})+(n-k)R^\omega(Y,\overline{Y}, Y, \overline{Y})\\
&\le& -(n-1)(k+1) n(n+1)\sigma \aint_{\mathbb{S}^{2n-1}} |Y|^4\\
&=&-(n-1)(k+1)\sigma\left( G^2 + |\lambda_\gamma|^4\right).
\end{eqnarray*}
Here $Y=\lambda_i w^i \frac{\partial}{\partial w^i}$ with respect to a normal frame $\{\frac{\partial}{\partial w^i}\}$, $\Ric^{\omega}$ and $R^{\omega}_{i\bar{j}k\bar{l}}$ are the Ricci and curvature tensor expressed with respect to the metric $\omega$ (namely the unitary frame $\{\frac{\partial}{\partial w^i}\}$.
The result then follows by identifying the terms invariantly.
\end{proof}
For the application it is useful to write (\ref{eq:22}) as
\begin{eqnarray}
2 \tilde{g}^{i \bar{j}} \tilde g^{k \bar{l}} \mathrm{R}_{i \bar{j} k \bar{l}}
&\leq&  \frac{-(n-1)(k+1)\sigma}{n-k}\left(G^2+|g|_{\tilde{g}}^{2}\right) -2\frac{k-1}{n-k} G \cdot \mathrm{tr}_{\tilde{g}} \mathrm{Ric}\nonumber \\
&\quad& +\frac{k-1}{n-k}\left( \langle \Ric, \tilde{\omega}\rangle_{\tilde{g}}\langle \omega, \tilde{\omega}\rangle_{\tilde{g}}-\langle \omega, \Ric\rangle_{\tilde{g}}\right). \label{eq:useful}
\end{eqnarray}

\section{Proof of Theorem \ref{thm:main1}}
Assume that  the canonical line bundle $K_N$ of $(N, \omega)$ is  not nef.  Then there exists $\epsilon_0>0$ such that $\epsilon_0[\omega]-C_1(N)$ is nef but not K\"ahler.  Thus,  $\forall \epsilon>0$, $(\epsilon+\epsilon_{0})[\omega]-C_{1}(N)$ is K\"ahler.  This means that there exists  a smooth function $\phi_{\epsilon}$ such that
\begin{equation}\label{eq:defKM}
\omega_\epsilon:={} (\epsilon_{0}+\epsilon) \omega - \mathrm{Ric}(\omega)+\sqrt{-1} \partial \bar{\partial} \phi_{\epsilon}>0.
\end{equation}

By Aubin-Yau's existence theorem and  a prior estimate for a complex Monge-Amp\`ere equation, we first prove the theorem below.

\begin{theorem}\label{nef}
Let $(N, \omega)$ be a compact K\"ahler manifold which satisfies (\ref{k-ricci}) for some $\sigma\ge0$. Then $K_N$ is nef.
\end{theorem}

\begin{proof}For any  $\epsilon > 0$, we consider the  complex Monge-Amp\`ere equation for $\psi_\epsilon$,
\begin{equation}\label{MA-equ}
\left((\epsilon+\epsilon_{0}) \omega-\mathrm{Ric}(\omega)+\sqrt{-1} \partial \bar{\partial} (\phi_{\epsilon}+\psi_{\epsilon})\right)^{n}
={} e^{\phi_{\epsilon}+\psi_{\epsilon}+\frac{k-1}{2(n-k)}(\phi_{\epsilon}+\psi_{\epsilon})} \omega^{n}
\end{equation}
and
\begin{align}\label{positive}(\epsilon+\epsilon_{0}) \omega-\mathrm{Ric}(\omega)+\sqrt{-1} \partial \bar{\partial} (\phi_{\epsilon}+\psi_{\epsilon})>0.
\end{align}
By the Aubin-Yau theorem \cite{Aub, Yau},  there is a unique solution $\psi_\epsilon$ of (\ref{MA-equ}).
For simplicity,  let
\begin{equation}
\begin{aligned}
\tilde{\omega}_{\epsilon}:={} & (\epsilon+\epsilon_{0}) \omega-\mathrm{Ric}(\omega)+\sqrt{-1} \partial \bar{\partial} (\phi_{\epsilon}+\psi_{\epsilon})\\
={} &\omega_\epsilon+\sqrt{-1}\partial\bar{\partial} \psi_\epsilon,\\
u_{\epsilon}:={} & \phi_{\epsilon}+\psi_{\epsilon}.
\end{aligned}
\end{equation}
Then taking $\partial\bar{\partial} \log(\cdot )$ on both sides of  \eqref{MA-equ},  we  see that  (\ref{MA-equ}) is equivalent  to
\begin{equation}\label{tile-Ricci}
\begin{aligned}
\widetilde{\mathrm{Ric}}_{\epsilon}:=&\mathrm{Ric}(\tilde{\omega}_{\epsilon})
={}  \mathrm{Ric}(\omega)-\sqrt{-1} \partial \bar{\partial} \left(u_{\epsilon}+\frac{k-1}{2(n-k)}u_{\epsilon}\right)\\
={} & -\tilde{\omega}_{\epsilon}+(\epsilon+\epsilon_{0}) \omega- \sqrt{-1} \frac{k-1}{2(n-k)} \partial \bar{\partial}u_{\epsilon}.
\end{aligned}
\end{equation}

Let $G=G_\epsilon=\mathrm{tr}_{\tilde{\omega}_\epsilon}\omega$. Then by the calculation in the proof of the Schwarz Lemma \cite{Yau-sch},  and in particular (2.3) of \cite{Ni-cpam} (also see computations in the earlier work of \cite{Chern, Lu}), as well as the $C^2$-estimate computations in \cite{Aub, Siu} (a slight different calculation was done in \cite{Yau, Tian}),  we have that
\begin{equation}\label{schwarz-lemma}
\tilde{\Delta}_{\epsilon}(\log G) \geq \frac{1}{G} \left(\widetilde{\mathrm{Ric}}_{\epsilon_ {p \bar{q}}} {\tilde{g}_{\epsilon}}^{p \bar{j}} {\tilde{g}_{\epsilon}}^{i \bar{q}} g_{i \bar{j}}-{\tilde{g}_{\epsilon}}^{i \bar{j}} {\tilde{g}_{\epsilon}}^{k \bar{l}}\mathrm{R}_{i\bar{j}k\bar{l}}\right).
\end{equation}
Applying  Lemma \ref{Riemann-curvature-estimate} (namely (\ref{eq:useful}) to $G=G_\epsilon$ we have the estimate
\begin{equation}
\begin{aligned}\label{curvature-2}
 \frac{1}{G}\tilde{g}_{\epsilon}^{i \bar{j}} \tilde{g}_{\epsilon}^{k \bar{l}} \mathrm{R}_{i \bar{j} k \bar{l}}
\leq {} & \frac{-(n-1)(k+1)\sigma}{2(n-k)}\left(G+\frac{1}{G}|g|_{\tilde{g}_{\epsilon}}^{2}\right)
-\frac{k-1}{n-k}\mathrm{tr}_{\tilde{g}_{\epsilon}} \mathrm{Ric}\\
&+\frac{k-1}{2(n-k)} \frac{1}{G} \left( G \cdot \mathrm{tr}_{\tilde{g}_{\epsilon}} \mathrm{Ric} - \left\langle \omega, \mathrm{Ric} \right\rangle_{\tilde{g}_{\epsilon}} \right).
\end{aligned}
\end{equation}
Choosing local coordinates such that $(\tilde{g}_{\epsilon})_{{i \bar{j}}}=\delta_{ij}, \quad g_{i \bar{j}}=g_{i \bar{i}} \delta_{ij}$, then we also have
\begin{equation}
\begin{aligned}
 G \cdot \mathrm{tr}_{\tilde{g}_{\epsilon}} \mathrm{Ric} - \left\langle \omega, \mathrm{Ric} \right\rangle_{\tilde{g}_{\epsilon}}
= {} & \sum_{i}\mathrm{Ric}_{i \bar{i}} \left( \sum_k g_{k \bar{k}}-g_{i \bar{i}}\right)\\
= {} &\sum_{i}\left(\mathrm{Ric}_{i \bar{i}} (\sum_{k\ne i} g_{k\bar{k}})\right)\\
\leq {} & \sum_{i}\left( (\epsilon+\epsilon_0)g_{i \bar{i}} +u_{\epsilon_{i \bar{i}}}\right) \left(\sum_k g_{k \bar{k}}-g_{i \bar{i}} \right)\\
= {} & (\epsilon+\epsilon_0)G^2-(\epsilon+\epsilon_0)|g|_{\tilde{g}_{\epsilon}}^{2}
+G \cdot \tilde{\Delta}_{\epsilon} u_{\epsilon}- \left\langle \sqrt{-1}\partial \bar{\partial} u, \omega \right\rangle_{\tilde{g}_{\epsilon}}.
\end{aligned}
\end{equation}
Here we used \eqref{positive} in the third line. Plugging this into (\ref{curvature-2}), we have
\begin{equation}\label{Riemann-curv}
\begin{aligned}
 -\frac{1}{G}\tilde{g}^{k\bar{l}}_{\epsilon}\tilde{g}^{i\bar{j}}_\epsilon R_{i\bar{j}k\bar{l}}
\geq {} & \frac{\sigma(n-1)(k+1)-(k-1)(\epsilon+\epsilon_0)}{2(n-k)}G\\
& + \frac{\sigma(n-1)(k+1)+(k-1)(\epsilon+\epsilon_0)}{2(n-k)}\frac{|g|^2_{\tilde{g}_\epsilon}}{G}+\frac{(k-1)}{(n-k)}\mathrm{tr}_{\tilde{g}_\epsilon}\mathrm{Ric} \\
& - \frac{k-1}{2(n-k)}\tilde{\Delta}_{\epsilon} u_{\epsilon}+\frac{k-1}{2(n-k)}\frac{1}{G}\left\langle \sqrt{-1}\partial\bar{\partial}u_\epsilon,\omega \right\rangle_{\tilde{g}_\epsilon}.
\end{aligned}
\end{equation}
On the other hand, a direct calculation using (\ref{tile-Ricci}) can express the first term in (\ref{schwarz-lemma}) as
\begin{equation}\label{Rici-curv}
\begin{aligned}
 \frac{1}{G}\widetilde{\mathrm{Ric}}_{\epsilon_{p\bar{q}}}\tilde{g}^{p\bar{j}}_{\epsilon}
 \tilde{g}^{i\bar{q}}_{\epsilon}g_{i\bar{j}}
= &\frac{1}{G}\left\langle\widetilde{\mathrm{Ric}}_{\epsilon},\omega \right\rangle_{\tilde{g}_\epsilon}\\
= &\frac{1}{G}\left\langle-\tilde{w}_{\epsilon}+(\epsilon+\epsilon_0)\omega-\sqrt{-1}\partial\bar{\partial}(\frac{(k-1)}{2(n-k)} u_\epsilon), \omega \right\rangle_{\tilde{g}_\epsilon}\\
={} &\frac{1}{G}\left\langle-\tilde{w}_{\epsilon}+(\epsilon+\epsilon_0)\omega, \omega \right\rangle_{\tilde{g}_\epsilon}-\frac{1}{G}\frac{(k-1)}{2(n-k)}\left\langle \sqrt{-1} \partial\bar{\partial}u_\epsilon,\omega \right\rangle_{\tilde{g}_\epsilon}.
\end{aligned}
\end{equation}
Note that we used \eqref{tile-Ricci} in the third line above.
Combining \eqref{schwarz-lemma}, \eqref{Riemann-curv} and \eqref{Rici-curv}, we have
\begin{align}\label{laplace-log-G-1}
 \tilde{\Delta}_{\epsilon}(\log G)\notag
\geq {} & \frac{\sigma(n-1)(k+1)-(k-1)(\epsilon+\epsilon_0)}{2(n-k)}G\notag\\
& +\frac{\sigma(n-1)(k+1)+(k-1)(\epsilon+\epsilon_0)}{2(n-k)}\frac{|g|^2_{\tilde{g}_{\epsilon}}}{G}\notag\\
& +\frac{(k-1)}{(n-k)}\mathrm{tr}_{\tilde{g}_\epsilon}\mathrm{Ric}-\frac{(k-1)}{2(n-k)}\tilde{\Delta}_{\epsilon}u_{\epsilon}\notag\\
&+\frac{1}{G}\left\langle -\tilde{\omega}_{\epsilon}
+(\epsilon+\epsilon_0)\omega,\omega \right\rangle_{\tilde{g}_\epsilon}.
\end{align}
Hence
\begin{align}\label{laplace-log-G}
 \tilde{\Delta}_{\epsilon}\left(\log G-\frac{k-1}{2(n-k)} u_\epsilon\right)\notag
\geq {} & \frac{\sigma(n-1)(k+1)-(k-1)(\epsilon+\epsilon_0)}{2(n-k)}G\notag\\
& +\frac{\sigma(n-1)(k+1)+(k-1)(\epsilon+\epsilon_0)}{2(n-k)}\frac{|g|^2_{\tilde{g}_{\epsilon}}}{G}\notag\\
& +\frac{(k-1)}{(n-k)}\left(\mathrm{tr}_{\tilde{g}_\epsilon}\mathrm{Ric}-\tilde{\Delta}_{\epsilon}u_{\epsilon}\right)\notag\\
&+\frac{1}{G}\left\langle -\tilde{\omega}_{\epsilon}
+(\epsilon+\epsilon_0)\omega,\omega \right\rangle_{\tilde{g}_\epsilon}.
\end{align}
Next, we observe that
\begin{eqnarray*}
|g|^2_{\tilde{g}_\epsilon}&\geq& \frac{G^2}{n},\\
-\tilde{\Delta}_\epsilon u_\epsilon &=&-\tilde{g}_\epsilon^{i\bar{j}}\left(\tilde{g}_{\epsilon_{i\bar{j}}}
+\mathrm{Ric}_{i\bar{j}}-(\epsilon+\epsilon_0)g_{i\bar{j}}\right)\\
&=&-n-\mathrm{tr}_{\tilde{g}_\epsilon}\mathrm{Ric}+(\epsilon+\epsilon_0)G, \mbox { and }\\
\frac{1}{G}\left\langle (\epsilon+\epsilon_0)\omega-\tilde{w}_\epsilon, \omega \right\rangle_{\tilde{g}_\epsilon}&=&\frac{1}{G}(\epsilon+\epsilon_0)|g|^2_{\tilde{g}_\epsilon}-\frac{1}{G}G\geq-1.
\end{eqnarray*}
Plugging these three inequalities/equation above  into \eqref{laplace-log-G}, we see that
\begin{equation}
\begin{aligned}\label{eq:3-15}
&\tilde{\Delta}_{\epsilon}\left(\log G-\frac{(k-1)}{2(n-k)}u_{\epsilon}\right)\\
\geq {} & \left(\frac{n}{n}\cdot\frac{\sigma(n-1)(k+1)-(k-1)(\epsilon+\epsilon_0)}{2(n-k)}\right)G
+\left(\frac{(\epsilon+\epsilon_0)(k-1)2n}{(n-k)2n}\right)G\\
&+\left(\frac{\sigma (n-1)(k+1)+(k-1)(\epsilon+\epsilon_0)}{2(n-k)n}\right)G-1-\frac{n(k-1)}{n-k}\\
={} & \left(\frac{(n+1)\sigma(n-1)(k+1)}{2(n-k)n}+\frac{(\epsilon+\epsilon_0)(k-1)(n+1)}{2(n-k)n}\right) G-1-\frac{n(k-1)}{n-k}\\
\geq {} & \max\left\{\frac{(n+1)\sigma(n-1)(k+1)}{2(n-k)n},\frac{(\epsilon+\epsilon_0)(k-1)(n+1)}{2(n-k)n}\right\}\cdot   G-1-\frac{n(k-1)}{n-k}.\\
\end{aligned}
\end{equation}

Now we apply the maximum principle to get a lower estimate of $\tilde{\omega}_\epsilon$. At  the maximum of $u_\epsilon$, say $x_0$, since $\sqrt{-1}\partial\bar{\partial} u_\epsilon \le 0$ we have
 that $((\epsilon + \epsilon_0)\omega- \Ric(\omega))(x_0) \ge \tilde{\omega}_\epsilon>0$ and
\begin{eqnarray*}
e^{\frac{2n-k-1}{2(n-k)}\sup_N  u_\epsilon} &=& e^{\frac{2n-k-1}{2(n-k)}u_\epsilon(x_0)}\\
& \le&\left.\frac{ ((\epsilon + \epsilon_0)\omega - \Ric(\omega))^n}{\omega^n}\right|_{x=x_0}\\
&\le& C,\end{eqnarray*}
 for some $C$ independent of $\epsilon$.
This proves a uniform upper bound for $u_\epsilon$, and
hence that
\begin{equation}\label{eq:volupper}
\sup_N \frac{\tilde{\omega}_\epsilon^n}{\omega^n}\le C, \mbox{ equivalently } W_n\ge C^{-1}, \mbox{ with } W_n:=\frac{\omega^n}{\tilde{\omega}_\epsilon^n}.
\end{equation}
Again we apply the maximum principle to $\log G-\frac{(k-1)}{2(n-k)}u_{\epsilon}$. By (\ref{eq:3-15}), at the point $x_0'$, where the  maximum of $\log G-\frac{(k-1)}{2(n-k)}u_{\epsilon}$ is attained,
we have that
\begin{equation}\label{upper-bound}
G(x_0')\leq C \text{ for some C independent of } \epsilon.
\end{equation}
Since $G W_n^{\frac{k-1}{2n-k-1}}= G e^{ -\frac{k-1}{2(n-k)} u_\epsilon}$, we infer that $\sup_{N} \left( G  W_n^{\frac{k-1}{2n-k-1}}\right)$ is also attained at $x_0'$.  By GM-AM inequality  $G\cdot W_n^{\frac{k-1}{2n-k-1}} \le G \left(\frac{G}{n}\right)^{\frac{k-1}{(2n-k-1)n}}$ at $x_0'$.
This, together with (\ref{upper-bound}), implies $ \sup_{N} \left(G  W_n^{\frac{k-1}{2n-k-1}}\right)\le C$ for some $C>0$  independent of $\epsilon$.
Combining this with  (\ref{eq:volupper}) we have that
\begin{equation}\label{eq:keyc2} G\le C\, \mbox { hence } \tilde{\omega}_\epsilon\ge A\omega,\end{equation}
for a constant $A>0$ independent of $\epsilon$. This is a contradiction to that $\epsilon_0[\omega]-C_1(N)$ is not K\"ahler by taking $\epsilon\to 0$. This completes the proof of Theorem  \ref{nef}.
\end{proof}

A remark is appropriate to compare the above proof with that of \cite{Wu-Yau}. The idea of using an Aubin-Yau solution  is the same. The difference lies in the details. First we came up with a modified  Monge-Amp\`ere equation to accommodate the new curvature condition. Secondly Wu-Yau's proof \cite{Wu-Yau} of the $C^2$-estimate can be obtained  by a direct application of Royden's version of Yau's Schwarz lemma (precisely, Theorem 1, p554 of \cite{Roy}). Namely no additional proof is necessary  for bounding $G$ under the assumption of \cite{Wu-Yau} (namely the holomorphic sectional curvature $H<0$), in view of an obvious lower bound on $\Ric(\tilde{\omega})$ from (\ref{tile-Ricci}). By comparison, some nontrivial manipulations are needed above (at the least to  the best knowledge of the authors) to get the $C^2$-estimate since one can not infer any useful information from  (\ref{schwarz-lemma}) directly under $\Ric_k<0$ for some $k>1$.

Once the nefness of $K_N$ is established, the ampleness of $K_N$ follows as Theorem 7 of \cite{Wu-Yau} provided that $\sigma>0$. In this case we take $\epsilon_0=0$. By considering  the Monge-Amp\`ere equation (\ref{MA-equ}), repeating the argument above, since $\sigma>0$ is assumed now,   we still can have the uniform estimates (\ref{eq:volupper}), (\ref{eq:keyc2}) and the upper bound of $u_\epsilon$ from the key estimate (\ref{eq:3-15}) independent of $\epsilon$. Moreover
 the elementary inequality
$\operatorname{tr}_\omega \tilde{\omega}_\epsilon \le \frac{1}{(n-1)!}(\operatorname{tr}_{\tilde{\omega}_\epsilon} \omega)^{n-1}\frac{\tilde{\omega}_\epsilon^n}{\omega^n}$ implies that $\operatorname{tr}_\omega \tilde{\omega}_\epsilon \le C$. Hence we have that $\tilde{\omega}_\epsilon$ and $\omega$ are equivalent. Namely for some $C>0$ independent of $\epsilon$
\begin{equation}\label{eq:equiv}
C^{-1} \omega \le \tilde{\omega}_\epsilon\le C \omega.
\end{equation}
This also gives the $C^0$-estimate (namely the lower bound of $u_\epsilon$) by the equation (\ref{MA-equ}). The  $C^3$-estimate of Calabi \cite{Aub, Yau, Tian} also applies here (cf.  \cite{Tosatti-Yang} for an adapted calculation to a  settings similar to (\ref{MA-equ})).  Alternatively one can also use the $C^{2, \alpha}$-estimate of Evans as in \cite{Siu}. Uniform estimates for up to the third order derivatives of $u_\epsilon$ allow one to apply the Arzela-Ascoli compactness to get a convergent subsequence out of $u_\epsilon$ as $\epsilon \to 0$.

Taking $\epsilon\to 0$, and letting \begin{eqnarray*}
u_\infty&:=&\lim_{\epsilon\to 0} u_\epsilon,\\
\omega_\infty &:=&-\Ric(\omega)+\sqrt{-1}\partial\bar{\partial} u_\infty>0,\end{eqnarray*}then it is easy to see that  (\ref{MA-equ}) becomes
$$
\left(-\Ric(\omega)+\sqrt{-1}\partial \bar{\partial} u_\infty \right)^n=e^{u_\infty +\frac{k-1}{2(n-k)}u_\infty} \omega^n.
$$
Taking $\partial \bar{\partial} \log(\cdot)$ on both sides of the above equation we have that
$$
\Ric(\omega_\infty)=-\omega_\infty -\frac{k-1}{2(n-k)}\sqrt{-1}\partial\bar{\partial} u_\infty.
$$
This implies that $K_N$ is ample. The existence of a K\"ahler-Einstein metric is known by Aubin-Yau's theorem.

We also remark that the argument can be easily modified to prove the same result under the assumption:
$$
\alpha |X|^2\Ric(X, \overline{X})+\beta R(X, \overline{X}, X, \overline{X})\le -\sigma |X|^4, \forall  X \mbox{ of } (1, 0)\mbox{-type},
$$
for some positive constants $\alpha, \beta$ and $\sigma\ge 0$.
The existing literatures (e.g. \cite{Wu-Yau-2}) is enough to extend Theorem \ref{thm:main1} to the case that $\Ric_k$ is quasi-negative (as well as $\sigma$ above is quasi-positive) proving that $K_N$ is big. We leave the details to interested readers.

We are
 grateful to Professor M$^{\mbox{c}}$Kernan for showing us an example of a smooth algebraic variety $N^n$ with ample canonical line bundle, which  admits a linear hypersurface $\mathbb{P}^{n-1}$. Together with Theorem 1.3 of \cite{Ni-cpam}, the example and Theorem \ref{thm:main1}
  show that the class of manifolds with a K\"ahler metric of $\Ric_k<0$, for $1\le k\le n-1$, is a strictly smaller class than that of manifolds with $c_1<0$ (equivalently those admitting a K\"ahler metric with $\Ric<0$). Note that on this example with $\mathbb{P}^{n-1}$, by Theorem 1.3 of \cite{Ni-cpam} a K\"ahler metric with $S_{n-1}\le 0$ is not possible either. By taking product with a curve of high genus repeatedly this gives an example of K\"ahler manifold which has $\Ric_k<0$, but $S_{k-1}>0$ somewhere, since if $S_{k-1}\le 0$ everywhere, Theorem 1.3 of \cite{Ni-cpam}  implies the impossibility of an embedded $\mathbb{P}^{k-1}$ in such manifold. This picture contrasts sharply with the fact that the class of K\"ahler manifolds with $\Ric_k>0$ (for some $k<n$) is  strictly larger  than that of the Fano manifolds \cite{Ni}. It also suggests an interesting question, namely if a compact K\"ahler manifold with $\Ric_k<0$  admits a K\"ahler metric with $\Ric_{k+\ell}<0$ for $\ell\ge 1$ (with $k+\ell<n$ since the case $k+\ell=n$ has been proven)?

 \section{Appendix}

In this appendix first we show that the averaging technique in Section 2 (cf. Appendix of \cite{Ni-cpam}, which was suggested by F. Zheng) also gives a quick proof of a result of Demailly-Skoda \cite{DS} on the relation between the  Nakano positivity and Griffiths positivity of holomorphic vector bundles. The original proof used an action of $\mathbb{Z}_q^r$ (Royden used a similar action in \cite{Roy}).

\begin{proposition}[Demailly-Skoda]\label{prop:DS} Let $(E, h)$ be a holomorphic vector bundle with $\operatorname{rank}(E)=r$ over a complex manifold $N^n$. Let $(\det(E), \det(h))$ be the determinant line bundle. Assume that $(E, h)$ is Griffiths positive. Then $E\otimes \det(E)$ (equipped with the induced metric) is Nakano positive.
\end{proposition}
Before the proof recall that Nakano positivity means that for any section nonzero section $\tau=\sum_{i=1}^r\sum_{\alpha=1}^n \tau^{i\alpha} \frac{\partial}{\partial z^\alpha} \otimes e_i$ (abbreviated as $\tau^{i\alpha} \frac{\partial}{\partial z^\alpha} \otimes e_i$) of $T'N\otimes E$,
$$
\sum_{\alpha, \beta=1}^n \sum_{i, k=1}^r\Theta_{\alpha\bar{\beta}i\bar{k}} \tau^{i\alpha}\overline{\tau^{k\beta}}>0
$$
where $\Theta$ denotes the curvature of $E$ with
$\Theta_{\alpha\bar{\beta}i\bar{k}}=\langle \Theta_{\frac{\partial}{\partial z^\alpha} \frac{\partial}{\partial \bar{z}^\beta}} (e_i), \overline{e_k}\rangle$.
 Below we assume that $\{\frac{\partial}{\partial z^\alpha}\}$ and $\{e_i\}$ are normal frames at a point $p\in N$.
\begin{proof} By direct calculation on the metric of the tensor product, the Nakano positivity of  $E\otimes \det(E)$ amounts to (Einstein convention applied) showing that for any $\tau\ne 0$
\begin{equation}\label{eq:Nak}
\Theta_{\alpha \bar{\beta} k\bar{k}}\tau^{i\alpha}\overline{\tau^{i\beta}}+\Theta_{\alpha\bar{\beta}i\bar{k}} \tau^{i\alpha}\overline{\tau^{k\beta}}>0.
\end{equation}
For a section $\tau$ as above and $w=(w^1, \cdots, w^r)\in \mathbb{S}^{2r-1}\subset E_p$ (identified as $\mathbb{C}^r$), let $W=\sum_{\alpha=1}^n \left(\sum_{i=1}^r \tau^{i\alpha}w^{i}\right)\frac{\partial}{\partial z^\alpha}$ and $u=\sum_{k=1}^r\bar{w}^ke_k$ be elements of $T'_pN$ and $E_p$. The Griffiths positivity implies that $\langle \Theta_{W\overline{W}}(u), \bar{u}\rangle>0$ for generic $w\in \mathbb{S}^{2r-1}$.  As in \cite{Ni-Zheng2}, taking the integration average $\aint$ over $\mathbb{S}^{2r-1}$, the Berger's lemma implies that
\begin{eqnarray*}
r(r+1)\aint \langle \Theta_{W\overline{W}}(u), \bar{u}\rangle\, d\mu(w)&=&r(r+1)\aint \Theta_{\alpha \bar{\beta} l\bar{k}}\tau^{i\alpha} w^i\overline{\tau^{j\beta} w^j} \bar{w}^k w^l\, d\mu(w)\\
&=& \sum_{i\ne k} \Theta_{\alpha\bar{\beta}k\bar{k}} \tau^{i\alpha}\overline{\tau^{i\beta}}+\sum_{i\ne j} \Theta_{\alpha\bar{\beta}i\bar{j}} \tau^{i\alpha}\overline{\tau^{j\beta}}+2\Theta_{\alpha\bar{\beta}i\bar{i}} \tau^{i\alpha}\overline{\tau^{i\beta}}\\
&=& \Theta_{\alpha \bar{\beta} k\bar{k}}\tau^{i\alpha}\overline{\tau^{i\beta}}+\Theta_{\alpha\bar{\beta}i\bar{k}} \tau^{i\alpha}\overline{\tau^{k\beta}}.
\end{eqnarray*}
This proves (\ref{eq:Nak}), hence the proposition.
\end{proof}

Secondly we include M$^{\mbox{c}}$Kernan's construction of the algebraic manifold mentioned in the previous section. The result and the proof are all due to him.

\begin{proposition}[M$^{\mbox{c}}$Kernan]\label{prop:james} Fix a positive integer $n$.
There is a smooth projective variety X of dimension n with the following two properties
\begin{itemize}\item[(1)] $K_X$ is ample, and
\item[(2)] $X$ contains a copy of $\mathbb{P}^{n-1}$.
\end{itemize}
\end{proposition}

\begin{lemma}\label{lem:02}
 Let $X$ be the cone  over  the $d$-uple embedding of $\mathbb{P}^{n-1}$ in $\mathbb{P}^N$.\footnote{Here we take the cone over a hyperplane section in $\mathbb{P}^{N+1}$. }  Let $\pi: Y \rightarrow X$ be the blow up of the vertex $p$, let $E$ be the
exceptional divisor, let $H \subset E \simeq \mathbb{P}^{n-1}$
be a hyperplane, let $D = K_Y|_E$,
the restriction of the canonical divisor of $Y$ to $E$ and let $m=d-n$.
We have
\begin{itemize}
\item[(1)]  $K_Y = \pi^* K_X-\frac{m}{d}E$ and
\item[(2)] $D = mH$.
\end{itemize}
In particular, we have
\begin{itemize}
\item[(i)] If $d < n$ then $-D$ is ample.
\item[(ii)] If $d = n$ then $D$ is numerically trivial.
\item[(iii)] If $d > n$ then $D$ is ample.
\end{itemize}
\end{lemma}
\begin{proof}
 For (1), we start with the equation
$$K_Y + E = \pi^* K_X + aE$$
where the rational number $a$, known as the log discrepancy, is to be
determined. If we restrict both sides to $E$ we get
\begin{eqnarray*} -nH &=& K_E\\
&=& \left.(K_Y + E)\right|_E\\
&=& \left.(\pi^*K_X + aE)\right|_E\\
&=& \pi^*\left.K_X\right|_E + \left. aE\right|_E\\
&=& -ad\, H.
\end{eqnarray*}
Here the first line is the usual formula for the canonical divisor of
projective space and we apply adjunction to get from the first line to
the second line.
It follows that $a =\frac{n}{d}$.
This gives (1) and restricting to $E$ gives (2). \end{proof}
Now we prove Proposition \ref{prop:james}. \begin{proof} We start with $W \subset \mathbb{P}^{N+1}$, the closure  of the cone from Lemma \ref{lem:02}, for any $d > n$.
Then $W$ is a projective variety with an isolated singularity $p$. Pick
an ample divisor $H$. Let $\pi : V \rightarrow W$ blow up the point $p$. By Lemma \ref{lem:02}
we have $
K_V = \pi^*K_W - \frac{m}{d} E$. Fix a positive integer $k$ and consider the divisor
$K_W + kH$.
Let $G = \pi^*H$. We have
\begin{eqnarray*}
K_V + kG &=& \pi^* (K_W + kH) - \frac{m}{d}E\\
&=& \pi^*(kH) + \pi^*K_W - \frac{m}{d}E.
\end{eqnarray*}
As $\pi^*K_W -\frac{m}{d} E$ is relatively ample, it follows that $K_V + kG$ is ample
if $k$ is sufficiently large (cf. \cite{Hartshorne}, Proposition (7.10.b) of Chapter II).

Pick $B \in |2kH|$, a general element of the linear system $|2kH|$. Then
Bertini implies that $B$ is a smooth divisor that does not contain $p$. The
$\mathbb{Q}$-divisor
$\frac{1}{2}B -H$
defines a double cover $\sigma : Y\rightarrow W$ with branch locus $B$ (cf. \cite{KM}).
Then $Y$ has two isolated singularities $q$ and $r$ lying over $p$ and is
otherwise smooth. Both singularities are analytically isomorphic to
the cone singularity of Lemma \ref{lem:02}.

The Riemann-Hurwitz formula reads as
$$K_V =\sigma^*
(K_W + \frac{1}{2}B).$$
As
$$ K_W +\frac{1}{2}B\sim_{\mathbb{Q}} K_W + kH$$
is ample, it follows that $K_V$ is ample, as $\sigma$ is finite.

Let $\psi: X\rightarrow Y$ be the blow up of $q$ and $r$. Then $X$ is a smooth
projective variety and
$$K_X = \psi^* K_Y -\frac{m}{d}\left(E_q + E_r\right),$$
where $E_q$ is the exceptional divisor over $q$ and $E_r$ is the exceptional
divisor over $r$. Note that $X$ is also double cover $\tau : X\rightarrow V $ of $V$
branched over the divisor
$C - 2H$,
where $C = \pi^*B$ is the strict transform of $B$. Observe that via the
commutative diagram
\[
\begin{tikzcd}
X \arrow[r,"\psi"] \arrow[d,swap,"\tau"] &
  Y \arrow[d,"\sigma"] \\
V \arrow[r,"\pi"] & W
\end{tikzcd}
\]
we have
\begin{eqnarray*}K_X &=& \psi^*K_Y -\frac{m}{d}\left(E_q+E_r\right)\\
&=& \psi^*\sigma^*
(K_W +\frac{1}{2}B) -\frac{m}{d}
(E_q + E_r)\\
&=& \tau^*\pi^*(
K_W +\frac{1}{2}B)-\frac{m}{d} \tau^* E\\
&=& \tau^*(\pi^*
(K_W +\frac{1}{2}B) -\frac{m}{d} E).
\end{eqnarray*}
We already saw that
$$\pi^*
(K_W +\frac{1}{2}B)-\frac{m}{d} E$$
is ample, for $k$ sufficiently large. It follows that $K_X$ is ample as $\tau$ is a
finite morphism. This proves the claim (1).

On the other hand both $E_q$ and $E_r$ are copies of projective space and
this gives  (2).
\end{proof}

Let $\mathcal{M}_{n, k}^{-}$ be the set of $n$-dimensional compact  manifolds with a K\"ahler metric such that its $\Ric_k<0$.  Let $\mathcal{M}_n^-$ be the set of $n$-dimensional compact  manifolds with ample canonical line bundle.
Let $\mathcal{S}_{n, k}^-$ be the set of $n$-dimensional compact  manifolds with a K\"ahler metric such that its $k$-th scalar $S_k<0$. Clearly $\mathcal{M}_{n, k}^{-}\subset \mathcal{S}_{n, k}^-$.
\begin{corollary} The following relation holds:
\begin{itemize}
\item[(1)]
$\mathcal{M}_{n, k}^{-}\subsetneq \mathcal{M}_{n}^{-}$, $\quad  \forall\,  1\le k<n, $
and
\item[(2)]  $\mathcal{M}_{n, k}^{-}\nsubseteq  \mathcal{S}_{n, k-1}^-, \forall\,  2\le k\le n.$
\end{itemize}
\end{corollary}
\begin{proof} The result follows by combining Theorem 1.3 of \cite{Ni-cpam}, the example above and Theorem \ref{thm:main1}, after taking a suitable product with high genus curves multiple times.
\end{proof}

The  relation (2)  for $k=2$ in particular implies that Theorem \ref{thm:main1} provides a new result beyond what the main result of \cite{Wu-Yau} can possibly cover.

\section*{Acknowledgments}{}

The first author's research is supported by China Postdoctoral Grant No. BX$20200356$.
The research of the second author is partially supported by  ``Capacity Building for Sci-Tech Innovation-Fundamental Research Funds".

\end{document}